\documentclass{amsart}
\usepackage{amsfonts}
\usepackage{amsmath,amscd}
\usepackage{amssymb}
\usepackage{amsthm}
\usepackage{newlfont}
\newcommand{\f}{\frac}

\newcommand{\ds}{\displaystyle}

 \newtheorem{thm}{Theorem}[section]
 \newtheorem{cor}[thm]{Corollary}
 \newtheorem{lem}[thm]{Lemma}
 \newtheorem{prop}[thm]{Proposition}
 \theoremstyle{definition}
 \newtheorem{defn}[thm]{Definition}
 \theoremstyle{remark}

 \newtheorem{ex}[thm]{Example}
 
 \numberwithin{equation}{section}

\begin{document}

\title[The exterior degree of a pair of finite groups]
 {The exterior degree of a pair of finite groups}

\author[P. Niroomand]{Peyman Niroomand}
\address{School of Mathematics and Computer Science\\
Damghan University, Damghan, Iran}
\email{p$\_$niroomand@yahoo.com, niroomand@du.ac.ir}
\author[R. Rezaei]{Rashid Rezaei}
\address{Department of Mathematics, Faculty of Sciences, Malayer University, Malayer,
Iran}

\thanks{\textit{Mathematics Subject Classification 2010.}  20F99, 20P05.}


\keywords{ Exterior degree, commutativity degree, Schur multiplier, capability of groups, pair of groups.}

\date{\today}


\begin{abstract}
The exterior degree of a pair of finite groups $(G,N)$, which is a generalization of the exterior degree of finite groups, is the probability for two elements $(g,n)$ in $(G,N)$ such that
$g\wedge n=1$. In the present paper, we state some relations between this concept and the relative commutatively degree, capability and the Schur multiplier of a pair of groups.

\end{abstract}

\maketitle

\section{Introduction}



Let $G$ be a group with a normal subgroup $N$, then $(G,N)$ is said to be a pair of groups. Let $G$ and $N$ acting on each other and themselves by
conjugation, remember that \cite{loday, brow} the non-abelian tensor product $G\otimes N$ is  the group
generated by the symbols $g\otimes n$ subject to the relations
\begin{equation*}
gg^{'}\otimes n=(~^gg^{'}\otimes ~^gn)(g\otimes n),
\end{equation*}
\begin{equation*}
g\otimes nn^{'}=(g\otimes n)(~^ng\otimes ~^nn^{'}),
\end{equation*}
for all $g, g^{'}$ in $G$ and $n,n^{'}$ in $N$.\\
As it mentioned in \cite{loday2}, a relative
central extension of the pair $(G,N)$ consists of a group homomorphism $\sigma: M\rightarrow G$ together with an action of $G$ on $M$, such that:
\begin{itemize}
\item[(i)] $\sigma(M) = N$;
\item[(ii)]  $\sigma(^gm) = g^{-1}\sigma(m)g$, for all $g\in G$ and $m\in M$;
\item[(iii)] $\ds^{\sigma(m_1)} m = m_1^{-1}m m_1$, for all $m,m_1 \in M$;
\item[(iv)]  $ker\sigma \subseteq Z_G(M),$
\end{itemize}
in which
\[Z_G(M)=\{m\in M ~|^gm=m ~\text{for all g in}~G\}.\]
The pair $(G, N)$ is called capable if it admits a central extension such that $ker\sigma=Z_G(M)$.
In particular if $M=N$ and $G$ acts on $N$ by conjugation,  $Z_G(N)$ is denoted by $Z(G,N)$.

The exterior product $G\wedge N$ is obtained from $G\otimes N$ by imposing the additional relation  $x\otimes x$ for all $x$ in $N$ and the image of $g\otimes n$ is denoted by $g\wedge n$ for all $g\in G, n \in N$.
By using the notations in \cite{Ellis2}, the exterior $G$-centre of $N$ is a central subgroup of $N$ which is defined as follows \[Z^{^\wedge} (G,N)=\{n\in N~|~1=g\wedge n\in G\wedge N~\text{for all g in}~ G \}.\] Already \cite[Theorem 3]{Ellis2} shows $Z^{^\wedge} (G,N)$ allows us to decide
when $(G,N)$ is capable. More precisely a pair $(G,N)$ is capable if and only if $Z^{^\wedge} (G,N)=1$. It can be checked that
$Z^{^\wedge} (G,N)=\ds\bigcap_{x\in G}{}_NC^{^\wedge}(x)$, in which ${}_NC^{^\wedge}(x)$ is the exterior
centralizer of $x$ in the pair of groups $(G,N)$ and it is equal to the set of all elements
$n\in N$ such that $1=x\wedge n\in G\wedge N $. Also for all $x\in N$ we denote ${}_GC^{^\wedge}(x)$ the set of all elements $g\in G$ such that $1=g\wedge x\in G\wedge N$.  In the case for which $N=G$ and $x\in G$, $Z^{^\wedge} (G,G)=Z^{^\wedge} (G)$ and ${}_GC^{^\wedge}(x)=C_{G}^{^\wedge}(x)$ are the exterior centralizer of $x$ in $G$ and the exterior centre of $G$, respectively.

The commutator map $\kappa: G\wedge N \rightarrow[G,N]$ is the group homomorphism defined on the generators by
 $g\wedge n\mapsto [g,n]=gng^{-1}n^{-1}$ $(g\in G, n\in N)$. The Kernel of $\kappa$ is the Schur multiplier of a pair of groups, $\mathcal{M}(G,N)$, which was stated by Brown, Loday \cite{loday} and Ellis in \cite{Ellis}.
\section{Known results on the relative commutativity degree and exterior degree}
In this section, some known results are remembered for the relative commutativity degree and exterior degree.
For any pair of finite groups $(G,N)$, we use the notation of  \cite{Erfanian} and define the commutativity degree of pair as follows
 \[\ds d(G,N)=\f1{|G||N|} |\{(x,y)\in G\times N~|~xy=yx\}|.\]
Obviously, if $G=N$, then $d(G,G)=d(G)$, the commutativity degree of a finite group $G$, and $d(G,N)=1$ if and only if $N\subseteq Z(G)$.
\begin{thm}\label{le2}\cite[Theorem 3.9]{Erfanian}
Let $H$ and $N$ be two subgroups of $G$ such that $N\unlhd G$ and $N\subseteq H.$
Then
\[d(G,H)\leq d(G/N,H/N)d(N),\] and if $N\cap [H, G]=1$, then the equality holds.
\end{thm}

The concept of exterior degree of finite group, $d^{\wedge}(G)$, is defined in  \cite{pr} as the probability for two elements $g$ and $g'$ in $G$ such that $g\wedge g'=1$.

It is seen that if  $k(G)$ is the number of conjugacy classes of $G$, then
$|G|^2\ds d^{^\wedge}(G)=\sum_{i=1}^{k(G)}\sum_{x\in C_i}|C_G^{^\wedge}(x)|,$   where $C_1,...,C_{k(G)}$ are the conjugacy classes of $G$.

The following two inequalities in \cite{pr} will be generalized in the next section.
\begin{thm}\label{pr1}\cite[Theorem 2.3]{pr}For every finite group $G$,
\[\f{d(G)}{|\mathcal{M}(G)|}+\f{|Z^{^\wedge}(G)|}{|G|}(1-\f{1}{|\mathcal{M}(G)|})\leq d^{^\wedge}(G)\leq d(G)-(\f{p-1}p)(\f{|Z(G)|-|Z^{^\wedge}(G)|}{|G|}),\]
where $p$ is the smallest prime number dividing the order of
$G$.\end{thm}
In the case for which $G$ is a capable group, the above upper bound for $d^{^\wedge}(G)$  can be improved as follows.
\begin{thm}\label{pr2}\cite[Theorem 2.8]{pr} Let $G$ be a non-abelian capable group and $p$ be the smallest prime number dividing the order of
 $G$, then $d^{^\wedge}(G)\leq1/p.$\end{thm}
\section{Exterior degree of a pair of finite groups}
In this section, we will define the exterior degree of a pair finite group $(G,N)$.  In the special case when $N=G$, the concept of exterior degree of $G$ is obtained.

\begin{defn}Let $(G,N)$ be a  pair of finite groups, set \[C:=|\{(g,n)\in G\times N~ |~ 1=g\wedge n\in G\wedge N\}|.\] We define the  exterior degree of the pair of groups $(G,N)$ as the ratio  \[d^{^\wedge}(G,N)=\ds\frac{|C|}{|G||N|}.\] Obviously, if $N=G$, then $d^{^\wedge}(G,G)=d^{^\wedge}(G)$ and $d^{^\wedge}(G,N)=1$ if and only if $N\subseteq Z^{\wedge}(G,N)$.
\end{defn}
Let a finite group $H$ acting on a finite group $K$ by conjugation. Then we denote the number of conjugacy $H$-classes of $K$ by $k_H(K)$.

The following lemmas are useful in the future investigation.

\begin{lem}\label{l1}
Let $(G,N)$ be a pair of finite groups, then \[d^{^\wedge}(G,N)=\ds\frac{1}{|G||N|}\sum_{x\in G}| {}_NC^{^\wedge}(x)|=\ds\frac{1}{|G||N|}\sum_{x\in N}|{}_GC^{^\wedge}(x)|.\]
\end{lem}
\begin{proof}
The proof is straightforward.
\end{proof}

\begin{lem}\label{l2} Let $(G,N)$ be a pair of finite groups. Then
\begin{itemize}\item[(i)] If $k_N(G)=k$ and $\{x_1,...,x_k\}$ is a system of representatives for conjugacy $N$-classes of a group $G$, then
\[d^{^\wedge}(G,N)=\f{1}{|G|}\sum_{i=1}^{k}\f{|{}_NC^{^\wedge}(x_i)|}{|C_N(x_i)|}.\]
\item[(ii)] If $k_G(N)=t$ and $\{x_1,...,x_t\}$ is a system of representatives for conjugacy $G$-classes of a group $N$, then
\[d^{^\wedge}(G,N)=\f{1}{|N|}\sum_{i=1}^{t}\f{|{}_GC^{^\wedge}(x_i)|}{|C_G(x_i)|}.\]
\end{itemize}
\end{lem}
 \begin{proof}(i) Let $C_1,...,C_k$ be the conjugacy $N$-classes of $G$ and  $x_i\in C_i$ for $1\leq i\leq k$.
For every $y\in C_i$, there exists $n\in N$ such that $y=x_i^n$, hence
$|{}_NC^{^\wedge}(y)|=|{}_NC^{^\wedge}(x_i)|.$ Therefore
\[\begin{array}{lcl}
|G||N|d^{^\wedge}(G,N)&=&\ds\sum_{x\in G}|{}_NC^{^\wedge}(x)|=\sum_{i=1}^{k}\sum_{x\in C_i}|{}_NC^{^\wedge}(x)|=\ds\sum_{i=1}^{k}[N:C_N(x_i)]|{}_NC^{^\wedge}(x_i)|\vspace{.2cm}\\
&=&\ds|N|\sum_{i=1}^{k}\f{|{}_NC^{^\wedge}(x_i)|}{|C_{N}(x_i)|},
\end{array}\]
as asserted.

$(ii)$ The proof is similar to the pervious part.
\end{proof}
\begin{prop}
Let $(G_1,N_1)$ and $(G_2,N_2)$ be two pairs of finite groups such that $G=G_1\times G_2$ and $N=N_1\times N_2$ in where  $|G_1|$ and $|G_2|$ are coprime. Then \[d^{^\wedge}(G_1\times G_2,N_1\times N_2)=d^{^\wedge}(G_1,N_1)d^{^\wedge}(G_2,N_2).\]
\end{prop}
\begin{proof}
\[\begin{array}{lcl}
d^{^\wedge}(G,N)&=&\ds\frac{1}{|G||N|}\sum_{(x,y)\in N}|{}_GC^{^\wedge}((x,y))|\vspace{.3cm}\\&=&\ds\frac{1}{|G_1||G_2||H_1||H_2|}
\sum_{x\in N_1}\sum_{y\in N_2}|{}_{G_1}C^{^\wedge}(x)||{}_{G_2}C^{^\wedge}(y)|\vspace{.3cm}\\
&=&\ds\frac{1}{|G_1||H_1|}\sum_{x\in N_1}|{}_{G_1}C^{^\wedge}(x)|\ds\frac{1}{|G_2||H_2|}\sum_{y\in N_2}|{}_{G_2}C^{^\wedge}(y)|
\vspace{.3cm}\\
&=&d^{^\wedge}(G_1,N_1)d^{^\wedge}(G_2,N_2).
\end{array}\]
\end{proof}
\begin{lem} For all $x$ in $G$, the factor group $C_N(x)/{}_NC^{^\wedge}(x)$ is isomorphic to a subgroup of $\mathcal {M}(G,N).$
\end{lem}
\begin{proof}
For all $x\in G$, we can define  $f:C_N(x)\rightarrow \mathcal {M}(G,N)$ by $y\mapsto y\wedge x$. Since the elements of $\mathcal{M}(G,N)$ are fixed under the action of $G$, $f$ is actually a homomorphism with ${}_NC^{^\wedge}(x)$ as its kernel.  \end{proof}
 The following theorem give a generalization of Theorem \ref{pr1}.
\begin{thm}\label{a}Let $(G,N)$ be a pair of finite groups and $p$ be the smallest prime number dividing the order of $G$. Then
\begin{itemize}
\item[(i)] $\ds\f{d(G,N)}{|\mathcal{M}(G,N)|}+\f{|Z^{^\wedge}(G,N)|}{|G|}(1-\f{1}{|\mathcal{M}(G,N)|})\leq d^{^\wedge}(G,N);$
\item[(ii)] $\ds d^{^\wedge}(G,N)\leq d(G,N)-(\f{p-1}p)(\f{|Z(G,N)|-|Z^{^\wedge}(G,N)|}{|N|}).$
\end{itemize}
\end{thm}
\begin{proof} $(i)$ By using Lemma \ref{l2} $(i)$,
\[\begin{array}{lcl}
d^{^\wedge}(G,N)&=&\ds\f{1}{|G|}\sum_{i=1}^{k}\f{|{}_NC^{^\wedge}(x_i)|}{|C_N(x_i)|}\vspace{.3cm}\\
&\geq&\ds\f1{|G|}\left(|Z^{^\wedge}(G,N)|+\f1{|\mathcal{M}(G,N)|}(k-|Z^{^\wedge}(G,N)|)\right)\vspace{.3cm}\\
&=&\ds\f{k}{|G||\mathcal{M}(G,N)|}+\f{|Z^{^\wedge}(G,N)|}{|G|}\left(1-\f1{|\mathcal{M}(G,N)|}\right).
\end{array}\]The results is obtained by the fact that $d(G,N)=k/|G|$.

$(ii)$ If $x\notin$ $Z^{^\wedge}(G,N)$, then $[G:{}_GC^{^\wedge}(x)]\geq p$. Hence by Lemma \ref{l2} $(ii)$,
\[\begin{array}{lcl}
d^{^\wedge}(G,N)&=&\ds\f1{|N|}\sum_{i=1}^{t}\f{|{}_GC^{^\wedge}(x_i)|}{|C_G(x_i)|}\vspace{.3cm}\\
&\leq&\ds\f{|Z^{^\wedge}(G,N)|}{|N|}+\f1p\f{|Z(G,N)|-|Z^{^\wedge}(G,N)|}{|N|}+\f{t-|Z(G,N)|}{|N|}\vspace{.3cm}\\
&=&\ds d(G,N)-(\f{p-1}p)(\f{|Z(G,N)|-|Z^{^\wedge}(G,N)|}{|N|}).
\end{array}\]
\end{proof}

\begin{thm}
If $p$ is the smallest prime number dividing the order of $G$, then for every  pair of finite groups $(G,N)$, \[ \ds\f{|Z^{^\wedge}(G,N)|}{|N|}+\f{p(|N|-|Z^{^\wedge}(G,N)|)}{|G||N|}\leq d^{^\wedge}(G,N)
\leq \ds\f{|Z^{^\wedge}(G,N)|}{|N|}+\f{|N|-|Z^{^\wedge}(G,N)|}{p~|N|}.\]
\end{thm}
\begin{proof} By using Lemma \ref{l1},
\[\begin{array}{lcl}d^{^\wedge}(G,N)&=&\ds\frac{1}{|G||N|}\sum_{x\in N}|{}_GC^{^\wedge}(x)|\vspace{.3cm}\\&=&\ds\f{1}{|G||N|}\left(\sum_{x\in Z^{^\wedge}(G,N)}|{}_GC^{^\wedge}(x)|
+\sum_{x\notin Z^{^\wedge}(G,N)}|{}_GC^{^\wedge}(x)|\right).\end{array}\]
  Since $p\leq |{}_GC^{^\wedge}(x)|\leq |G|/p$ for all $x\notin Z^{^\wedge}(G,N)$, so
\[\begin{array}{lcl}\ds\f{1}{|G||N|}\big(\sum_{x\in Z^{^\wedge}(G,N)}|{}_GC^{^\wedge}(x)|
&+&\sum_{x\notin Z^{^\wedge}(G,N)}|{}_GC^{^\wedge}(x)|\big)\vspace{.3cm}\\&\leq&
\ds\f{|Z^{^\wedge}(G,N)|}{|N|}+\f{|N|-|Z^{^\wedge}(G,N)|}{p~|N|}.\end{array}\]
On the other hand,  \[\begin{array}{lcl}\ds\f{|Z^{^\wedge}(G,N)|}{|N|}&+&\ds\f{p(|N|-|Z^{^\wedge}(G,N)|)}{|G||N|}\vspace{.3cm}\\&\leq& \ds\f{1}{|G||N|}
\left(\sum_{x\in Z^{^\wedge}(G,N)}|{}_GC^{^\wedge}(x)|+\sum_{x\notin Z^{^\wedge}(G,N)}|{}_GC^{^\wedge}(x)|\right),\end{array}\]
hence the result follows.
\end{proof}
\begin{cor}
Let $(G,N)$ be a pair of finite groups and $p$ be the smallest prime divisor of the order of $G$. If   $N\neq Z^{^\wedge}(G,N)$, then
$d^{^\wedge}(G,N)\leq (2p-1)/p^2,$ In particular, $d^{^\wedge}(G,N)\leq\f 34.$
\end{cor}
\begin{proof}
In the case $N\nsubseteq Z(G)$, the  result is obtained by \cite[Theorem 3.9]{da}.
Now suppose that $N\subseteq Z(G)$.  Then $d(G,N)=1$, and so by Theorem \ref{a} $(ii)$,
\[d^{^\wedge}(G,N)\leq\ds 1-(\f{p-1}p)(1-\f{|Z^{^\wedge}(G,N)|}{|N|})\leq 1-(p-1)^2/p^2=(2p-1)/p^2.\]
\end{proof}
The example \ref{ex3} will show that the above upper bound is sharp and  the bound we obtained there is the best possible one.

\begin{lem}\label{gh} Let $(G,N)$ and $(G,H)$ be pairs of finite groups. Then the following sequence is exact.
\[N\wedge H\times G\wedge N\stackrel{l}\rightarrow G\wedge H\rightarrow G/N\wedge H/N\rightarrow 0. \]
\end{lem}
\begin{proof} The proof is similar to \cite[Proposition 9]{loday}.
\end{proof}
\begin{thm}Let $(G,N)$ and $(G,H)$ be pairs of finite groups. Then \[d^{^\wedge}(G,H)
\leq d^{^\wedge}( G/N,H/N),\] and  the equality holds
when $N\subseteq Z^{^\wedge}(G,H),$
\end{thm}

\begin{proof}
By using Lemma \ref{l1},
\[\begin{array}{lcl}
d^{^\wedge}(G,H)&=&\ds\f1{|G||H|}\ds\sum_{x\in G}|{}_{H}C^{^\wedge}(x)|=\f{1}{|G||H|}\sum_{xN\in \f GN}\sum_{n\in N}|{}_{H}C^{^\wedge}(xn)|
\vspace{.3cm}\\
&=&\ds\f{1}{|G||H|} \sum_{xN\in \f GN}\sum_{n\in N}|\f{{}_{H}C^{^\wedge}(xn)N}{N}||{}_{H}C^{^\wedge}(xn)\cap N|\vspace{.3cm}\\
&\leq&\ds\f1{|G||H|}\sum_{xN\in \f GN}\sum_{n\in N}|{}_{{\f HN}}C^{^\wedge}(xN)||{}_HC^{^\wedge}(xn)\cap N|\vspace{.3cm}\\ &=&
\ds\f1{|G||H|}\sum_{xN\in \f GN}|{}_{{\f HN}}C^{^\wedge}(xN)|\sum_{n\in N}|{}_HC^{^\wedge}(xn)\cap N|.\vspace{.3cm}\\
&\leq&\ds\f1{|G||H|}\sum_{xN\in \f GN}|{}_{\f HN}C^{^\wedge}(xN)||N|^2=d^{^\wedge}(\f GN,\f HN)
\end{array}\]
and so $d^{^\wedge}(G,H)\leq d^{^\wedge}(G/N,H/N).$

In the case for which $N\subseteq Z^{^\wedge}(G,H),$ Lemma \ref{gh} implies that $Im(l)=1$, and thus $G\wedge H\cong G/N\wedge H/N$ which implies that $d^{^\wedge}(G,H)
= d^{^\wedge}( G/N,H/N).$
\end{proof}

\begin{lem}\label{lg} Let $(G,N)$ be a pair of finite groups such that $N\nsubseteq Z(G)$, and $p$ be  the smallest prime number dividing the order of $G$, then
\begin{itemize}
\item[(i)] if $[G,N]\cap Z(G,N)=1$, then $d(G,N)\leq 1/p$;
\item[(ii)] $d(G,N)\leq 1/p-(1-1/p)\ds\f{|Z(G,N)|}{|N|}.$
\end{itemize}
\end{lem}
\begin{proof}$(i)$ Since the condition of equality in the Theorem \ref{le2} holds we have \[d(G,N)=d(G/Z(G,N),N/Z(G,N)).\]  Moreover $Z(G/Z(G,N),H/Z(G,N))=1$. Thus we can assume that $Z(G,N)=1$.

Let $d(G,N)=k_G(N)/|N|>1/p$, then \[|N|=1+\sum_{i=2}^{k_G(N)}[G:C_G(x_i)]\geq 1+p(k_G(N)-1)\geq1+|N|,\] which is a contradiction, hence the result follows.
 \proof$(ii)$ By using Lemma \ref{l1}, \[\begin{array}{lcl} d(G,N)&=&\ds\frac{1}{|G||N|}\sum_{x\in N}|C_G(x)|
\vspace{.3cm}\\ &=&\ds\frac{1}{|G||N|}\sum_{x\in Z(G,N)}|C_G(x)|+\ds\frac{1}
{|G||N|}\sum_{x\notin Z(G,N)}|C_G(x)|\vspace{.3cm}\\ &\leq&\ds\f{|G|}{|G||N|}|Z(G,N)|+
\ds\f{|G|}{p}(|N|-|Z(G,N)|)\vspace{.3cm}\\ &=&1/p-(1-1/p)\ds\f{|Z(G,N)|}{|N|}.\end{array}\] \end{proof}
Now it is easy to see that Theorem \ref{pr2} is obtained by the following theorem when $N=G$.
\begin{thm}\label{mt} Let $(G,N)$ be a pair of finite groups such that $N\nsubseteq Z(G)$,  and $p$ is the smallest prime number dividing the order of $G$. If the pair $(G,N)$ is capable, then
\[d^{^\wedge}(G,N)\leq 1/p .\]
\end{thm}
\begin{proof}
We can assume that $[G,N]\cap Z(G,N)\neq 1$, by Lemma \ref{lg}. Assume by contrary  for all $x$ in $N-Z(G,N)$ we have ${}_GC^{^\wedge}(x)=C_G(x),$ it follows that
\[Z(G,N)\leq\ds\bigcap_{x\in G-Z(G,N)}C_{G}(x)\cap N=\ds\bigcap_{x\in G-Z(G)}{}_{G}C^{^\wedge}(x)\cap N.\]
On the other hand, $\ds N\cap [G,N]\leq \bigcap_{x\in Z(G,N)}{}_{G}C^{^\wedge}(x)\cap N$ thus $Z(G,N)\cap [G,N]\leq Z^{^\wedge}(G,N)$, which is a contradiction. Hence Lemma \ref{l2} $(ii)$ follows
\[\begin{array}{lcl} d^{^\wedge}(G,N)&=& \ds\f{1}{|N|}\sum_{i=1}^{t}\f{|{}_{G}C^{^\wedge}(x_i)|}{|C_G(x_i)|}
\vspace{.3cm}\\&\leq& \ds \f1{|N|}(1+1/p(Z(G,H))-1+1/p+t-Z(G,N)-1)\vspace{.3cm}\\&=& \ds d(G,N)-(1-1/p)(\f{|Z(G,N)|}{|N|}).\end{array}\]
The rest of proof is obtained by Lemma \ref{lg} $(ii)$.

 \end{proof}
 Example \ref{ex4} will show that the above bound can be attained by a pair of capable groups, which shows the upper bound is sharp.
 \section{Some Examples}
 \begin{ex} \label{ex1} Consider the dihedral group $D_{2n}$ with the presentation \[D_{2n}=\langle a,b~|~a^{n}=1,~b^2=1,~ba=a^{-1}b\rangle.\]
 First assume that $n$ is odd.  $D_{2n}$ is the semidirect product of $\langle a \rangle$ by $C_2$.
 By using $(7)$ of \cite{Ellis} and \cite[Proposition 2.11.4]{kar}, we have $\mathcal{M}(D_{2n},\langle a \rangle)=1$.
Thus $d^{^\wedge}(D_{2n},\langle a \rangle)=d(D_{2n},\langle a \rangle)$ and \cite[Example 3.11]{Erfanian} states
$d(D_{2n},\langle a \rangle)=\ds\frac{n+1}{2n}.$

In the case $n$ even, $Z^{^\wedge}(D_{2n},\langle a \rangle)=1$ when $n>2$, we have
 \[\begin{array}{lcl} d^{^\wedge}(D_{2n},\langle a \rangle)&=&\ds\frac{1}{|D_{2n}||\langle a \rangle|}\sum_{x\in \langle a \rangle}|{}_{D_{2n}}C^{^\wedge}(x)|\vspace{.3cm}\\&=&
\ds\frac{2n+n^2/2+n/2(n/2-1)}{2n^2}\vspace{.3cm}\\&=&\ds\frac{n+1}{2n}.\end{array}\]
  \end{ex}
  \begin{ex}
  Let $Q_{2^n}$ be the generalized quaternion group of order $2^n$ with the presentation
  \[Q_{2^n}=\langle a,b~|~a^{2^{n-1}}=1,~a^{2^{n-2}}=b^2,~ba=a^{-1}b\rangle.\] Then
 \[\begin{array}{lcl} d^{^\wedge}(Q_{2^n},\langle a \rangle)&=&\ds\frac{1}{|Q_{2^n}||\langle a \rangle|}\sum_{x\in \langle a \rangle}|{}_{Q_{2^n}}C^{^\wedge}(x)|
  \vspace{.3cm}\\&=&\ds\frac{1}{2^n2^{n-1}}(2^n+(2^{n-1}-1)2^{n-1})\vspace{.3cm}\\&=&\ds\frac{2^{n-1}+1}{2^n}.\end{array}\]
  \end{ex}
\begin{ex}\label{ex3} Let $G= C_{4}$ and $N= 2C_{4}$. Then \cite[Corollary 8]{Ellis2} shows that the pair  $(C_{4},2C_{4})$ is capable, and so
\[\begin{array}{lcl} d^{^\wedge}(C_{4},2C_{4})&=&\ds\frac{1}{|2C_{4}||C_4|}\sum_{x\in C_4}|{}_{2C_{4}}C^{^\wedge}(x)|
  \vspace{.3cm}\\&=&\ds\frac{1}{8}(1+2+1+2)=\f34.\end{array}\]
\end{ex}

\begin{ex}\label{ex4}
Let $G=D_8$ and $N$ be the subgroup $\langle a^2,ab\rangle$ of $G$. A computation similar to Example \ref{ex1} shows that the pair $(D_8,\langle a^2,ab\rangle)$ is capable, and
\[\begin{array}{lcl}d^{^\wedge}(D_{8},\langle a^2,ab\rangle)&=&\ds\frac{1}{|\langle a^2,ab\rangle||D_8|}\sum_{x\in D_8}|{}_{\langle a^2,ab\rangle}C^{^\wedge}(x)|
  \vspace{.3cm}\\&=&\ds\frac{1}{32}(4+2+1+2+2+2+1+2)=\f12.\end{array}\] We remark that all these computations
can be done by using the HAP package \cite{eh}.
\end{ex}

\end{document}